\documentclass{AIMS}
\usepackage{amsmath}
  \usepackage{paralist}
  \usepackage{graphics} 
  \usepackage{epsfig} 
\usepackage{graphicx}  \usepackage{epstopdf}
 \usepackage[colorlinks=true]{hyperref}
\hypersetup{urlcolor=blue, citecolor=red}

\def\currentvolume{X}
 \def\currentissue{X}
  \def\currentyear{200X}
   \def\currentmonth{XX}
    \def\ppages{X--XX}
     \def\DOI{10.3934/xx.xx.xx.xx}

\newtheorem{Theorem}{Theorem}[section]
\newtheorem*{Theorem A}{Theorem A}
\newtheorem*{Theorem A'}{Theorem A'}

\newtheorem*{Conj*}{Conjecture}

\newtheorem{Definition}[Theorem]{Definition}
\newtheorem{Proposition}[Theorem]{Proposition}
\newtheorem{Lemma}[Theorem]{Lemma}

\newtheorem{Remark-numbered}[Theorem]{Remark}
\newtheorem{Remarks-numbered}[Theorem]{Remarks}

\newtheorem*{Claim}{Claim}
\newtheorem*{Theorem B'}{Theorem B'}
\newtheorem{Claim-numbered}{Claim}

\title[No-shadowing for singular hyperbolic sets with a singularity] 
      {No-shadowing for singular hyperbolic sets with a singularity}

\author[Xiao Wen and Lan Wen]{}

\subjclass{Primary: 37C10, 37D30.}
 \keywords{shadowing property, singular hyperbolicity, star flows, multisingular hyperbolicity.}

 \email{wenxiao@buaa.edu.cn}
 \email{lwen@math.pku.edu.cn}

\thanks{The first author is supported by National Natural Science
Foundation of China (No. 11671025 and No. 11571188)
and the Fundamental Research Funds for the Central Universities.
The second author is
supported by National Natural Science
Foundation of China (No. 11231001). }

\thanks{}

\begin{document}
\maketitle

\centerline{\scshape Xiao Wen}
\medskip
{\footnotesize
 \centerline{School of Mathematical Sciences}
   \centerline{Beihang University, Beijing, 100191, China}
} 

\medskip

\centerline{\scshape Lan Wen }
\medskip
{\footnotesize
 \centerline{School of Mathematical Sciences}
   \centerline{Peking University, Beijing, 100871, China}
}

\bigskip


\begin{abstract}
We prove that every singular hyperbolic chain transitive set with a singularity does not admit the shadowing property.  Using this result we show that if a star flow has the shadowing property on its chain recurrent set then it satisfies Axiom A and the no-cycle conditions; and that if a multisingular hyperbolic set has the shadowing property then it is hyperbolic.
\end{abstract}

\section{Introduction}
The geometric Lorenz attractor is one of the most important examples for the flow theory. Its dynamics are rich and robust, but somewhat delicate. Indeed, it shares many important properties with a hyperbolic set while itself is not a hyperbolic set. For instance, it is robustly transitive, robustly with periodic orbits dense, etc., however, unlike a hyperbolic set, it does not admit the robust shadowing property. In fact it does not admit the (honest, non-robust) shadowing property  \cite{Kom}. In this article we show that this no-shadowing phenomenon appears for a wide class of ``nearly-hyperbolic" sets, namely, the so called {\it singular hyperbolic sets} (among which the geometric Lorenz attractor is a particular example), or more generally, the so called {\it multisingular hyperbolic sets}. We will see that the key reason that causes the no-shadowing phenomenon for these sets is the existence of a singularity. Indeed, we will show that every chain transitive multisingular hyperbolic set with a singularity does not admit the shadowing property. This highlights the striking and delicate difference between flows with a singularity and flows without singularities.

Let $M$ be a $d$-dimensional compact smooth Riemmanian manifold without boundary. Denote by $\mathcal{X}^1(M)$ the set of $C^1$ vector fields on $M$ endowed with the $C^1$ topology. For any $X\in\mathcal{X}^1(M)$, denote by $\varphi_t=\varphi_t^X: M\to M$ the flow generated by $X$. Let $\Phi_t=\Phi^X_t=D\varphi^X_t$ be the tangent flow generated by $\varphi_t$. First we state the definition of singular hyperbolic set.

A compact invariant set of $\varphi_t$ is called {\it positively singular hyperbolic} for $X$ if every singularity in $\Lambda$ is hyperbolic and if there is a $\Phi_t$ invariant splitting $T_\Lambda M=E^{ss}\oplus E^{cu}$ with two constants $C>1,\lambda>0$ such that the following three conditions are satisfied:

(i) $E^{ss}\oplus E^{cu}$ is a $(C,\lambda)$ dominated splitting, that is, $\|\Phi_t|_{E^{ss}(x)}\|\cdot\|\Phi_{-t}|_{E^{cu}(\varphi_t(x))}\|\leq Ce^{-\lambda t}$ for any $x\in\Lambda$ and any $t\geq0$;

(ii)  $E^{ss}$ is $(C,\lambda)$-contracting for $\Phi_t$, that is, $\|\Phi_t|_{E^{ss}(x)}\|\leq Ce^{-\lambda t}$ for any $x\in\Lambda$ and any $t\geq0$;

(iii) $E^{cu}$ is $(C,\lambda)$-area expanding, that is, $\|\wedge^2 \Phi_{-t}|_{E^{cu}(x)}\|<Ce^{-\lambda t}$ for any $x\in\Lambda$ and any $t\geq0$.

\noindent A compact invariant set $\Lambda$ is called {\it negatively singular hyperbolic} for $X$ if it is positively singular hyperbolic for $-X$. If $\Lambda$ is either negatively or positively singular hyperbolic for $X$, then we say that $\Lambda$ is {\it singular hyperbolic} of $X$.

The concept of singular hyperbolic set is introduced  by Morales, Pacifico and Pujals \cite{MPP} to characterize Lorenz-like dynamics. Briefly, it is a special kind of partially hyperbolic sets, very close to hyperbolic sets. In fact if a singular hyperbolic set contains no singularities then it reduces to a hyperbolic set. It is not a surprise that a singular hyperbolic set shares many properties with a  hyperbolic set, for instance the important expansiveness: it was proved that a singular hyperbolic attractor in $3$-dimensional manifold has the $K^*$-expansiveness (see \cite{APPV}) and any singular hyperbolic set has the rescaled expansiveness (see \cite{WW}). However, it does not share an important property of a hyperbolic set: the shadowing property.

Recall that, given $\delta>0$,  a sequence of $\{(x_i, t_i) | x_i\in M, t_i\geq 1\}_{i=a}^{b}~ (-\infty\leq a<b\leq+\infty)$ is called a $\delta$-pseudo orbit of $\varphi_t$ if $d(\varphi_{t_i}(x_i),x_{i+1})<\delta$ for all integer $a\leq i<b$. We say that an invariant set $\Lambda$ of $\varphi_t$ has the {\it shadowing property} or the {\it pseudo orbit tracing property} if for any $\varepsilon>0$ there is $\delta>0$ such that for any $\delta$-pseudo orbit  $\{(x_i, t_i) | x_i\in M, t_i\geq 1\}_{i=a}^{b}$ of $\varphi_t$ with all $x_i\in\Lambda$, there is a point $x\in M$ and an increasing homeomorphism $\theta:\mathbb{R}\to\mathbb{R}$ such that $d(\varphi_{t-s_i}(x_i), \varphi_{\theta(t)}(x))<\varepsilon$ for all $t\in[s_i,s_{i+1}]$ where $s_i$ is a sequence with $s_0=0$ and $s_i-s_{i-1}=t_i$ for all $i$. If, in addition, the shadowing points $x$ is contained in $\Lambda$, then we say ${\Lambda}$ has the {\it intrinsic shadowing property}.

Recall that a compact invariant set $\Lambda$ is called {\it chain transitive} if for any $x,y\in\Lambda$ and any $\delta>0$, there is a $\delta$-pseudo orbit $\{(x_i, t_i)\}_{i=1}^n$ in $\Lambda$ such that $x_1=x, x_n=y$. We say that a  compact invariant set is {\it nontrivial} if it does not reduce to a single orbit.

In this article, we prove that every nontrivial singular hyperbolic chain transitive set with a singularity does not admit the intrinsic shadowing property:
\vskip 0.2cm
{\bf Theorem A.}
{\it Let $\Lambda$ be a nontrivial singular hyperbolic chain transitive  set of $\varphi_t$. If ${\Lambda}$ has the intrinsic shadowing property, then it admits no singularity.}
\vskip 0.2cm

The concept of singular hyperbolic set is  closely related to the so called star flows.  Recall that a $C^1$ vector field $X$  is called a {\it star vector field} if there is a $C^1$ neighborhood $\mathcal{U}$ of $X$ in $\mathcal{X}^1(M)$ such that for any $Y\in\mathcal{U}$, every critical element $p$ of $Y$ is hyperbolic. Here by a {\it critical element} we mean either a singularity or a periodic orbit. Since a non-hyperbolic critical element can be easily turned into a hyperbolic one by an arbitrarily small perturbation of the vector field, a structurally stable vector field must have at least  the property that every critical element is hyperbolic. Thus the star property is just a simple necessary condition for a system to be structurally stable, and the set of star vector fields constitutes an important class of dynamical systems from the point of view of perturbations. A striking example of star flows is the geometric Lorenz attractor.

We give an application of Theorem A to star flows. The statement needs the notion of chain recurrent set of a vector field. Recall a point $x$ is called a {\it chain recurrent point} of $X$ if for any $\delta>0$, there exists a $\delta$-pseudo orbit $\{x_i\}_{i=0}^{n}$ with $n\geq1$ such that $x_0=x=x_n$. The set of chain recurrent points of $X$ is called the {\it chain recurrent set} of $X$, denoted by $CR(X)$. For any $x,y\in CR(X)$, we say that $x$ is {\it chain equivalent}, written $x\sim y$, if for any $\delta>0$, there exist two $\delta$-pseudo orbits $\{x_i\}_{i=0}^n ~(n\geq 1)$ and $\{y_i\}_{i=0}^m ~(m\geq 1)$ such that $x_0=x, x_n=y$ and $y_0=y, x_m=x$. The binary relation $\sim $ is an equivalent relationship on $CR(X)$. Every equivalent class of $\sim$ is called a {\it chain class} of $X$. It is easy to see that every chain  class is chain transitive. Given a critical element $p$, denote by $C(p)$ the chain class containing $p$. Here is an application of Theorem A to star flows. The statement concerns the standard notions of Axiom A and the no-cycle condition of S. Smale (definitions omitted).

\vskip 0.2cm
{\bf Theorem B.}
If $X$ is a star vector field such that the chain recurrent set $CR(X)$ has the shadowing property, then $X$ satisfies Axiom A and the no-cycle condition.
\vskip 0.2cm

Now we generalize Theorem A to the following Theorem C that refers to the so called multisingular hyperbolic sets introduced by Bonatti and da Luz \cite{Bl}, which generalizes the notion  of singular hyperbolicity and characterizes the star flows.  The  precise definition of multisingular hyperbolicity will be given in Section 4. Note that characterizing star flows  has been a long standing problem proposed by Liao and Ma\~n\'e in the late seventies of the last century. For flows without singularities, the star property is equivalent to Axiom A plus the no-cycle condition \cite {GW}. For flows with singularities, in dimension 3 and on a chain class $\Lambda$, the star property is generically equivalent to singular hyperbolicity \cite {Y, MP}. In higher dimensions and on a chain class $\Lambda$, the star property is genericlly  equivalent to singular hyperbolicity  as long as singularities of $\Lambda$ have the same index \cite {SGW}.  Because the notion of singular hyperbolicity forces singularities in a chain class to have the same index, to use singular hyperbolicity to characterize the star flows,  the only point that was left in dark is whether every chain class $\Lambda$ of a generic star flow must have singularities of the same index. Surprisingly, da Luz \cite {Bl} found a striking example of star flow with a chain class that has, robustly, singularities of different indices. This means that, for the characterization of (generic) star flows, the notion of singular hyperbolicity is too restrictive. Bonatti and da Luz \cite {Bl} then introduced a notion that relaxes singular hyperbolicity, called multisingular hyperbolity, that allows in particular different indices of singularities in a chain class, and prove that the star property is, on a chain class, generically equivalent to multisingular hyperbolicity. Thus the long standing problem of characterizing the star flows eventually reached a generic answer, formulated in terms of the new notion of multisingular hyperbolic sets.
 Since then quite some results that hold for singular hyperbolic sets are found to hold also for multisingular hyperbolic sets. The next result generalizes Theorem A from singular hyperbolic sets to multisingular hyperbolic sets.

\vskip 0.2cm
{\bf Theorem C.}
Let $\Lambda$ be a nontrivial multisingular hyperbolic chain transitive set of $\varphi_t$. If $\Lambda$ has the intrinsic shadowing property, then $\Lambda$ contains no singularity.
\vskip 0.2cm

The proofs for Theorem A, B and C will be given in the following three sections, respectively.

\section{The Proof of Theorem A}

 Let $\sigma\in M$ be a hyperbolic singularity of $X$. We call the dimension of its stable space the {\it index} of $\sigma$, denoted $Ind(\sigma)$. Denote by $W^s(\sigma)$  and $W^u(\sigma)$ the stable and unstable manifolds for  $\sigma$. As usual, for any $x\in M$ and $r>0$, denote by $B(x,r)$ the $r$-ball centered at $x$.

The following lemma is standard. For completeness we include the proof here.

\begin{Lemma}\label{lem1}
Let $\Lambda$ be a nontrivial chain transitive set and $\sigma\in\Lambda$ be a hyperbolic singularity, then $(W^s(\sigma)\setminus\{\sigma\})\cap\Lambda\neq\emptyset$ and $(W^u(\sigma)\setminus\{\sigma\})\cap\Lambda\neq\emptyset$.
\end{Lemma}

\begin{proof}
Since $\sigma$ is hyperbolic, we can find $\varepsilon_0>0$ such that, for any $z\in M$, if $d(\varphi_t(z),\sigma)\leq\varepsilon_0$ for all $t\geq 0$ then $z\in W^s(\sigma)$. Likewise, if $d(\varphi_{-t}(z),\sigma)\leq\varepsilon_0$ for all $t\geq 0$ then $z\in W^u(\sigma)$. For any $0<\varepsilon<\varepsilon_0$, denote by
$$W_{\varepsilon}^s(\sigma)=\{x : d(\varphi_t(x),\varphi_t(\sigma))\leq\varepsilon, \forall t>0\}.$$

\begin{Claim}
For any $0<\varepsilon<\varepsilon_0$, there is a point $x\in (B(\sigma, \varepsilon)\cap\Lambda)\setminus W_{\varepsilon}^s(\sigma)$.
\end{Claim}

In fact,
suppose for the contrary that there is $\varepsilon'>0$ such that $B(\sigma, \varepsilon)\cap\Lambda\subset W_{\varepsilon'}^s(\sigma)$.
We know that there is a $\delta>0$ such that
$$d(\bigcup_{t\geq1}\varphi_t(W_{\varepsilon'}^s(\sigma)),\partial (W_{\varepsilon'}^s(\sigma)))>\delta.$$
It is easy to see that $\delta <\varepsilon'$. By the assumption that $\Lambda$ is nontrivial we know that there is a point $y\in\Lambda$ which differs from $\sigma$. Without lost of generality we assume $d(\sigma, y)\geq \varepsilon'$. Since $\Lambda$ is chain transitive, there is a $\delta$-pseudo orbit $\{x_0,x_1,\cdots,x_n\}$ such that $x_0=\sigma$ and $x_n=y$. By the choice of $\delta$ we know that $d(x_1, \sigma)<\delta<\varepsilon'$. Since $B(\sigma, \varepsilon)\cap\Lambda\subset W_{\varepsilon'}^s(\sigma)$, it follows that $x_1\in W^s_{\varepsilon'}(\sigma)$. Because $d(x_2,\varphi_{t_1}(x_1))<\delta$ and $\varphi_{t_1}(x_1)\in\bigcup_{t\geq1}\varphi_t(W_{\varepsilon'}^s(\sigma))$ it follows that $x_2\in\Lambda\cap B_{\varepsilon'}(\sigma)$. Hence $x_2\in W^s_{\varepsilon'}(\sigma)$. Inductively, $x_i\in W^s_{\varepsilon'}(\sigma)$ for all $i=0,1,\cdots,n$. This contradicts  $y=x_n\notin B_{\varepsilon'}(\sigma)$ and proves the Claim.

By the Claim we can find a sequence $\{x_n\}$ with all $x_n\in\Lambda\setminus W^s_{\varepsilon_0}(\sigma)$ such that $x_n\to\sigma$ as $n\to\infty$. Since $x_n\notin W^s_{\varepsilon_0}(\sigma)$, for $n$ sufficiently large, there is $t_n$ such that $d(\varphi_{t_n}(x_n),\sigma)=\varepsilon_0$ and $d(\varphi_t(x),\sigma)<\varepsilon_0$ for all $t\in(0, t_n)$. It is easy to see that $t_n\to\infty$ as $n\to\infty$.
Let $y$ be an accumulation point of $\{\varphi_{t_n}(x_n)\}$, then $d(y,\sigma)=\varepsilon_0$, and $d(\varphi_t(y),\sigma)\leq\varepsilon_0$ for all $t\leq 0$. One can check that $y\in (W^u(\sigma)\setminus\{\sigma\})\cap\Lambda$. Similar arguments prove that $(W^s(\sigma)\setminus\{\sigma\})\cap\Lambda\neq\emptyset$.
\end{proof}

\begin{Lemma}\label{lem2}
Let $\Lambda$ be a positively singular hyperbolic set with splitting $T_\Lambda M=E^{ss}\oplus E^{cu}$, then $X(x)\in E^{cu}(x)$ for any $x\in\Lambda$.
\end{Lemma}

\begin{proof}
Assume the contrary that there is $x\in\Lambda$ such that $X(x)\notin E^{cu}(x)$. In particular, $\|X(x)\|\not= 0$. We consider the limit of $\Phi_{-t}(X(x))$ as $t\to+\infty$. Since $E^{ss}\oplus E^{cu}$ is a dominated splitting we know that the direction of $\Phi_{-t}(X(x))$ will converge into $E^{ss}(\varphi_t(x))$ as $t\to+\infty$. Since $E^{ss}$ is exponentially expanding by $\Phi_{-t}$, $\|\Phi_{-t}(X(x))\|$ will tend to the positive infinity as $t\to+\infty$. This contradicts  the fact that the norms $\|X(x)\|$ for $x\in \Lambda$ are bounded  on the compact set $\Lambda$.
\end{proof}

Since all singular hyperbolic sets $\Lambda$ considered below are  nontrivial, we will generally omit
the adjective ``nontrivial".

\begin{Lemma}\label{lem3}
Let $\Lambda$ be a positively singular hyperbolic chain transitive set with splitting $T_\Lambda M=E^{ss}\oplus E^{cu}$ and $\sigma\in\Lambda\cap Sing(X)$. If there is a point $x\in (W^s_\sigma\setminus\{\sigma\})\cap\Lambda$, then $\sigma$ has index $\dim(E^{ss})+1$.
\end{Lemma}
\begin{proof}
Let $\sigma\in\Lambda$ be a singulrity of $X$. Let $T_\sigma M=E^s_\sigma\oplus E^u_\sigma$ be the hyperbolic splitting of $\sigma$. Since $E^{ss}_\sigma$ is contracting, it follows that $E^{ss}_\sigma\subset E^s_\sigma$. Hence $Ind(\sigma)\geq\dim E^{ss}$.

Suppose $Ind(\sigma)=\dim E^{ss}$. Since for a fixed index the dominated splitting of $T_\sigma M$ is unique, we have
$E^{ss}_\sigma=E^s_\sigma$ and $E^{cu}_\sigma=E^u_\sigma$. By the assumption we know that there is a point $x\in (W^s_\sigma\setminus\{\sigma\})\cap\Lambda$. Note that $X(x)\not= 0$.  By the invariance of $W^s_\sigma$ we know that ${\rm Orb}(x)\subset W^s_\sigma$ and hence $X(x)\in T_x W^s_\sigma$. Since  $\Phi_t(T_x W^s_\sigma)$  converges to $E^{s}_\sigma=E^{ss}_\sigma$ as $t\to+\infty$, it follows that the limit set of $\Phi_t(<X(x)>)$ as $t\to \infty$  is contained in $E^{ss}_\sigma$. On the other hand, by Lemma \ref{lem2} we know that $X(x)\in E^{cu}(x)$ and hence the limit set of $\Phi_t(<X(x)>)$ as $t\to +\infty$  is contained in $E^{cu}_\sigma$. This contradiction proves that $Ind(\sigma)>\dim E^{ss}$.

Suppose $Ind(\sigma)\geq\dim E^{ss}+2$, then we can find two linearly independent  unit vectors $u, v\in E^{s}_\sigma\cap E^{cu}(\sigma)$. Since
$$\|\wedge^2\Phi_t(u\wedge v)\|\leq\|\Phi_t(u)\|\cdot\|\Phi_t(v)\|,$$
we can easily check that $\|\wedge^2\Phi_t(u\wedge v)\|\to 0$ as $t\to+\infty$. This contradict that $E^{cu}$ is area expanding. This proves that $Ind(\sigma)<\dim E^{ss}+2$, and hence $Ind(\sigma)=\dim(E^{ss})+1$.
\end{proof}

\begin{Lemma}\label{lem6}
Let $\Lambda$ be a chain transitive set of $X$ and $\sigma\in\Lambda$ be a hyperbolic singularity. If $\Lambda$ has the shadowing property, then for any neighborhood $U$ of $\Lambda$, there is a point $x\in(W^s(\sigma)\cap W^u(\sigma))\setminus \{\sigma\}$ such that $Orb(x)\subset U$. Moreover, if $\Lambda$ has the intrinsic shadowing property, then the homoclinic point $x$ can be chosen in $\Lambda$.
\end{Lemma}

\begin{proof}
Let $U$ be a given neighborhood of $\Lambda$. We can choose $\varepsilon>0$ small enough such that $d(\varphi_t(x),\varphi_t(\sigma))\leq \varepsilon$ for all $t\geq 0$ implies $x\in W^s(\sigma)$ and such that $d(\varphi_{-t}(y), \varphi_{-t}(\sigma))\leq\varepsilon$ for all $t\geq 0$ implies $y\in W^u(\sigma)$. We may assume that $\varepsilon>0$ is small enough so that for any $x\in\Lambda$, $d(x,y)<\varepsilon$ implies $\varphi_t(y)\in U$ for all $t\in[0,1]$.

By Lemma \ref{lem1}  there are two points $x_1\in (W^s(\sigma)\setminus\{\sigma\})\cap\Lambda$ and $x_2\in W^u(\sigma)\setminus\{\sigma\})\cap\Lambda$. Without loss of generality we assume that $\varepsilon<d(x_1,\sigma)$. By the assumption of the shadowing property, we can choose $\varepsilon/2>\delta>0$ such that any $\delta$-pseudo orbit can be $\varepsilon/2$-shadowed by a true orbit. By the chain transitivity of $\Lambda$ we can find a $\delta$-pseudo orbit  $\{(y_i,t_i)\}_{i=0}^n(n\geq1)$ in $\Lambda$ such that $y_0=x_2$ and $y_n=x_1$. Let $y_i=\varphi_{i}(x_1), t_i=1$ for $i\geq n$ and $y_i=\varphi_{-i}(x_2), t_i=1$ for $i\leq-1$. Then $\{(y_i,t_i):i\in\mathbb{Z}\}$ is a $\delta$-pseudo orbit. Let $s_i$ be the sequence with $s_0=0$ and $s_i-s_{i-1}=t_i$ for all $i\in\mathbb{Z}$. Since $\Lambda$ has the shadowing property, we can find a point $z\in\Lambda$ with an increasing homeomorphism $\theta:\mathbb{R}\to\mathbb{R}$ such that $d(\varphi_{t-s_i}(x_i),\varphi_{\theta(t)}(z))<\varepsilon/2$ for all $t\in[s_i, s_{i+1}]$. For this point $z$ we can check that $d(\varphi_t(z),\sigma)<\varepsilon$ and $d(\varphi_{-t}(z),\sigma)<\varepsilon$ for $t$ big enough. By the choice of $\varepsilon$ we know that $y\in W^s(\sigma)\cap W^u(\sigma)\setminus\{\sigma\}$ and $Orb(y)\subset U$. If $\Lambda$ has the intrinsic shadowing property, then the points $y$ can be chosen in $\Lambda$. This ends the proof of the lemma.
\end{proof}

In the rest of this section, let $\Lambda$ be a positively singular hyperbolic set with at least one singularity. By taking $T>0$ large we can find and fix $\eta\in(0,1)$ such that
$$\|\Phi_T|_{E^{ss}(x)}\|<\eta\ \ \ \ \ \text{ and }\ \ \  \|\Phi_T|_{E^{ss}(x)}\|\cdot\|\Phi_{-T}|_{E^{cu}(\varphi_T(x))}\|<\eta$$
for all $x\in\Lambda$. By the invariant manifolds theorem (see \cite{HPS} or \cite[Proposition2.3]{Man}), there are local stable manifolds associated with the subbundle $E^{ss}$, that is, there exists a family of $C^1$ closed disks $\{W_{loc}^{ss}(x)|x\in\Lambda\}$ which varies continuously with respect to $x$ such that
\begin{enumerate}
\item $T_xW_{loc}^{ss}(x)=E^{ss}(x)$ for any $x\in\Lambda$;
\item $\varphi_T(W_{loc}^{ss}(x))\subset W_{loc}^{ss}(\varphi_T(x))$ for any $x\in\Lambda$;
\item $\|\Phi_t|_{T_yW_{loc}^{ss}(x)}\|<\eta$ for all $x\in\Lambda$ and $y\in W_{loc}^{ss}(x)$.
\end{enumerate}
Item 3 guarantees that $d(\varphi_t(x),\varphi_t(y))\to 0$ for any $x\in\Lambda$ and $y\in W_{loc}^{ss}(x)$. For any $y\in\Lambda$ and $z\in W^{ss}_{loc}(y)$, denote by $d_s(y,z)$ the distance of $y,z$ in $W^{ss}_{loc}(y)$. By Item 1 one sees that
$$\frac{d_s(y,z)}{d(y,z)}\to 1 \text{\ \ \ \ as \ \ \ \ } d_s(y,z)\to 0.$$
In what follows we will use the notation $W_r^{ss}(x)$ to denote the disk $\{z\in W_{loc}^{ss}(x): d_s(y,z)<r\}$ for all $x\in\Lambda$. Let
$$W^{ss}(x)=\bigcup_{n=1}^{\infty}\varphi_{-nT}(W^{ss}_{loc}(\varphi_{nT}(x)))$$
be the {\it global} invariant manifold for $x\in\Lambda$. Then we have the following lemma which is \cite[Lemma 3.7]{SVY}.

\begin{Lemma}\label{lem4} The global invariant manifolds $W^{ss}(x)$ have the following properties:
\begin{enumerate}
\item $\varphi_t(W^{ss}(x))=W^{ss}(\varphi_t(x))$ for any $x\in\Lambda$ and $t\in\mathbb{R}$;
\item if $W^{ss}(x)\cap W^{ss}(y)\neq\emptyset$ then $W^{ss}(x)=W^{ss}(y)$ for any $x,y\in\Lambda$.
\end{enumerate}
\end{Lemma}

\begin{Lemma}\label{lem5}
$(W^{ss}(\sigma)\setminus\{\sigma\})\cap\Lambda=\emptyset$ for any singularity $\sigma\in\Lambda$.
\end{Lemma}
\begin{proof}
Assume the contrary that there is a point $x\in (W^{ss}(\sigma)\setminus\{\sigma\})\cap\Lambda$. Note that $X(x)\not= 0$.  By the invariance of $W^{ss}(\sigma)$ we know that ${\rm Orb}(x)\subset W^{ss}(\sigma)$ and hence $X(x)\in T_x W^{ss}(\sigma)=T_x(W^{ss}(x))$. By Item 1, $T_x(W^{ss}(x))=E^{ss}(x)$. Then $X(x)\in E^{ss}(x)$. On the other hand, by Lemma \ref{lem2} we have $X(x)\in E^{cu}(x)$. This is a contradiction, proving $(W^{ss}(\sigma)\setminus\{\sigma\})\cap\Lambda=\emptyset$.
\end{proof}

\begin{Lemma}\label{lem7}
Let $\sigma\in\Lambda$ be a singularity and $z\in W^s(\sigma)\cap\Lambda$. For any $y\in Orb(z)$, if $y\neq z$, then $W^{ss}_{loc}(y)\cap W^{ss}_{loc}(z)=\emptyset$.
\end{Lemma}
\begin{proof}
Assume the contrary that there are distinct points $y,z$ in the orbit of $z\in W^s(\sigma)$ with $W^{ss}_{loc}(y)\cap W^{ss}_{loc}(z)\neq\emptyset$. Without loss of generality, we can assume that $z=\varphi_\tau(y)$ with $\tau>0$. It is easy to see that  $W^{ss}_{loc}(y)\cap W^{ss}_{loc}(z)\neq\emptyset$ implies $W^{ss}_{loc}(\varphi_{nT}(y))\cap W^{ss}_{loc}(\varphi_{nT}(z))\neq\emptyset$ for all $n\geq0$.  Since $y,z\in W^s(\sigma)$, the two points $\varphi_t(y),\varphi_t(z)$ will be arbitrarily close to $\sigma$ for $t$ big enough. By Lemma \ref{lem5}, $y,z\notin W^{ss}(\sigma)$. Replacing $y,z$ by $\varphi_{nT}(y), \varphi_{nT}(z)$ if necessary, we can assume that $X(\varphi_t(y))$ is close to $E^c(\sigma)=E^{cu}(\sigma)\cap E^s_\sigma$ for all $t\in[0,\tau]$. Let $\gamma_0=\varphi_{[0,\tau]}(y)$. Let $\zeta\in W^{ss}_{loc}(y)\cap W^{ss}_{loc}(z)$. We can find two curves $\gamma_1\subset W^{ss}_{loc}(y)$ connecting $y, \zeta$ and $\gamma_2\subset W^{ss}_{loc}(z)$ connecting $z,\zeta$. Without loss of generality we assume that the tangent space of $\gamma_1$ (and $\gamma_2$ too) is close to $E^{ss}(\sigma)$. Denote by $l(\gamma)$ the length of the curve $\gamma$. By the domination of the splitting  $E^{ss}(\sigma)\oplus E^c(\sigma)$ we see that
$$\frac{l(\varphi_t(\gamma_1))}{l(\varphi_t(\gamma_0))}\to 0  \text{\ \ \ and\ \ \ }\frac{l(\varphi_t(\gamma_2))}{l(\varphi_t(\gamma_0))}\to 0$$
as $t\to+\infty$. Then  for $t$ big enough, $l(\varphi_t(\gamma_0))$ is close to $d(\varphi_t(y),\varphi_t(z))$, $l(\varphi_t(\gamma_1))$ is close to $d(\varphi_t(y),\varphi_t(\zeta))$, and $l(\varphi_t(\gamma_2))$ is close to $d(\varphi_t(z),\varphi_t(\zeta))$. Thus we have
$$\frac{d(\varphi_t(y),\varphi_t(\zeta))+d(\varphi_t(z),\varphi_t(\zeta))}{d(\varphi_t(y),\varphi_t(z))}\to0$$
as $t\to+\infty$. This is a contradiction.
\end{proof}

\begin{Proposition}\label{mainprop}
Let $\sigma$ be a singularity with a homoclinic point $x_0\in (W^s(\sigma)\cap W^u(\sigma))\setminus\{\sigma\}$. If $\Gamma=Orb(x_0)\cup\{\sigma\}$ is singular hyperbolic, then $\Gamma$ does not have the shadowing property.
\end{Proposition}
\begin{proof}
Without loss of generality we assume that $\Gamma$ is positively singular hyperbolic with singular hyperbolic splitting $E^{ss}\oplus E^{cu}$. Then $Ind(\sigma)=\dim(E^{ss})+1$ by Lemma \ref{lem3}. Denote by $W^{ss}(x)$ the strong stable manifold tangent to $E^{ss}(x)$ at any $x\in\Gamma$.
Since $\dim W^{ss}(\sigma)= \dim W^s_\sigma-1$, we can choose $\epsilon_0>0$ such that the locally stable manifold is given by
$$W_{\epsilon_0}^s(\sigma)=\{x : d(\varphi_t(x),\varphi_t(\sigma))\leq \epsilon_0, \forall t\geq 0\}.$$
Without loss of generality, we can assume that  $d(\varphi_t(x_0),\sigma)<\epsilon_0/2$ for all $t\geq 0$.

\begin{Claim-numbered}
There is $\varepsilon_1>0$ such that for any $\tau\in\mathbb{R}$, if $d(\varphi_\tau(x_0),x_0)<\varepsilon_1$ then $d(\varphi_t(x_0),\sigma)>\varepsilon_1$ for any $t$ in the interval between $0$ and $\tau$.
\end{Claim-numbered}

In fact, we can take $\varepsilon_1'>0$ such that $B(x_0,\varepsilon_1')\cap B(\sigma,\varepsilon_1')=\emptyset$. Let $\tau_1=\sup\{t>0: \varphi_t(x_0)\in B(x_0,\varepsilon_1')\}$ and $\tau_2=\inf\{t<0,\varphi_t(x_0)\in B(x_0,\varepsilon_1')\}$. Since $\varphi_{t\to\infty}(x_0)\to\sigma$, $\tau_1,\tau_2$ exist and $\tau_1<0$, $\tau_2>0$. Choose $0<\varepsilon_1<\varepsilon_1'$ such that $B(\sigma,\varepsilon_1)\cap\varphi_{[\tau_2,\tau_1]}(x_0)=\emptyset$. If $d(\varphi_\tau(x_0), x_0)<\varepsilon_1$, then $\tau\in[\tau_2,\tau_1]$ by the fact that $d(\varphi_\tau(x_0), x_0)<\varepsilon_1'$. By the choice of $\varepsilon_1$ we see that $d(\varphi_t(x_0),\sigma)>\varepsilon_1$ for any $t\in[0,\tau]$. This ends the proof of Claim 1.

\vskip 0.2cm

Since $W_{loc}^{ss}$ is expanding for $\varphi_{-T}$ and we have also an upper bound for $\|D\varphi^{-T}\|$, there is $r>0$ such that for any $x\in\Gamma$ and any $y\in W_r^{ss}(x)\setminus\{x\}$, one can find $n>0$ such that $\varphi_{-nT}(y)\in W_{loc}^{ss}(f^{-n}(x))\setminus W_r^{ss}(f^{-n}(x))$. Denoted $CW^s_r=\bigcup_{x\in \Gamma}(W^{ss}_{loc}(x)\setminus W_r^{ss}(x))$. From Lemma \ref{lem7} we know $\Gamma\cap CW^s_r=\emptyset$. Since $\Gamma, CW_r^s$ are compact, we can find a neighborhood $U$ of $\Gamma$ such that $U\cap CW_r^s=\emptyset$. There is $\varepsilon_2>0$ such that $B(y,\varepsilon_2)\subset U$ for all $y\in \Gamma$.

Note that $\dim W^{ss}(x)=\dim W^s(\sigma)-1$ for all $x\in\Gamma$ and $W^{ss}(x)$ varies continuously on $x$. By applying the Brouwer's invariance of domain theorem, for a given $\tau>0$, the set
$$\bigcup_{t\in(-\tau,\tau)}W_{r}^{ss}(\varphi_t(x_0))$$
forms a neighborhood of $x_0$ in $W_{\epsilon_0}^s(\sigma)$. Hence we can take $\varepsilon_3>0$ such that for any point $y\in W_{\epsilon_0}^s(\sigma)$ with $d(y,x_0)<\varepsilon_3$, one can find some $t\in\mathbb{R}$ such that $y\in W_{r}^{ss}(\varphi_t(x_0))$.

Let $\varepsilon_0=\min\{\epsilon_0/2, \varepsilon_1,\varepsilon_2,\varepsilon_3\}$. We will show that for any $\delta>0$, there is a $\delta$-pseudo orbit that can not be $\varepsilon_0$ shadowed by a true orbit.

Let $\delta>0$ be given. Without loss of generality we assume  $\delta<\varepsilon_0$. We can find a point $y_0$ in the negative orbit of $x_0$ such that  $d(y_0,\sigma)<\delta$, and then a positive number $t_0$ big enough such that $\varphi_{t_0}(y_0)$ is in the positive orbit of $x_0$ and  $d(\varphi_{t_0}(y_0),\sigma)<\delta$. Let $y_1=\sigma, t_1=1$ and $y_2=y_0, t_2=t_0$ and $y_i=\sigma, t_i=1$ for $i\neq 0,1,2$. One can easily check that $\{(y_i, t_i) :i\in\mathbb{Z}\}$ is a $\delta$-pseudo orbit. We will show that this pseudo orbit can not be $\varepsilon_0$ shadowed by a true orbit.

Assume the contrary that there is $z\in M$ and an increasing homeomorphism $\theta:\mathbb{R}\to\mathbb{R}$ such that $d(\varphi_{t-s_i}(y_i), \varphi_{\theta(t)}(z))<\varepsilon_0$ for all $t\in[s_i,s_{i+1}]$ where $s_i$ is a sequence with $s_0=0$ and $s_i-s_{i-1}=t_i$ for all $i$. One can easily check that $d(\varphi_{\theta(t)}(z),\sigma)<\varepsilon_0$ for any $t>2t_0+1$ and any $t<0$. By the choice of $\varepsilon_0\leq\epsilon_0$ we can see that $\varphi_{\theta(2t_0+1)}(z)\in W_{\epsilon_0}^s(\sigma)$ and $\varphi_{\theta(0)}(z)\in W_{\epsilon_0}^u(\sigma)$. This implies that $z\in W^{s}(\sigma)\cap W^u(\sigma)$.

\begin{Claim-numbered}
$z$ is not contained in $\Gamma$.
\end{Claim-numbered}

In fact, by the construction of the pseudo orbit,  there is $0<\tau_1<t_0$ such that $x_0=\varphi_{\tau_1}(y_0)$ and hence $x_0=\varphi_{\tau_1}(y_2)$ Thus $d(x_0, \varphi_{\theta(\tau_1)}(z))<\varepsilon_0$, $d(\sigma, \varphi_{\theta(t_0)}(z))<\varepsilon_0$, and $d(x_0, \varphi_{\theta(t_0+1+\tau_1)}(z))<\varepsilon_0$. Since $\theta$ is an increasing homeomorphism we know that $\theta(\tau_1)<\theta(t_0)<\theta(t_0+1+\tau_1)$. Since $d(x_0,\sigma)>\varepsilon_1\geq\varepsilon_0$, $z$ can not be $\sigma$ because the orbit of $z$ cross the $\varepsilon_0$ neighborhood of $x_0$. Now we prove that $z$ is not contained in the orbit of $x_0$. Assume that $z$ is in the orbit of $x_0$ and $z=\varphi_s(x_0)$. Then we can find $$\theta_1=s+\theta(\tau_1)<\theta_2=s+\theta(t_0)<\theta_3=s+\theta(t_0+1+\tau_1)$$ such that $d(\varphi_{\theta_1}(x_0),x_0)<\varepsilon_0, d(\varphi_{\theta_2}(x_0),\sigma)<\varepsilon_0, d(\varphi_{\theta_3}(x_0),x_0)<\varepsilon_0$. If $\theta_3\leq 0$, then $\theta_2\in[\theta_1,0]$. At the same time we have $d(\varphi_{\theta_2}(x_0),\sigma)<\varepsilon_0\leq\varepsilon_1$ and $d(\varphi_{\theta_1}(x_0),x_0)<\varepsilon_0\leq\varepsilon_1$, this contradicts Claim 1. If $\theta_1\geq 0$, then we have $0<\theta_2<\theta_3$ and $d(\varphi_{\theta_2}(x_0),\sigma)<\varepsilon_0$ and $d(\varphi_{\theta_3}(x_0),x_0)<\varepsilon_0$, also contradicting Claim 1. Similarly, if we have $\theta_1<0<\theta_3$, then either $\theta_2\in[\theta_1,0]$ or $\theta_2\in[0,\theta_3]$, contradicting Claim 1. This proves that $z$ is not in the orbit of $x_0$, proving Claim 2.

Note that we have chosen $x_0$ with the property $d(\varphi_t(x_0),\sigma)<\epsilon_0/2$ for all $t>0$, hence $d(\varphi_{\theta(t)}(z),\sigma)<\epsilon_0$ for all $t>t_0+1+\tau_1$. This implies that $\varphi_{\theta(t_0+1+\tau_1)}(z)\in W_{\epsilon_0}^s(\sigma)$. By the choice of $\varepsilon_0\leq\varepsilon_3$ we know that $\varphi_{\theta(t_0+1+\tau_1)}(z)\in W_r^{ss}(\varphi_t(x_0))$ for some $t\in\mathbb{R}$. By Claim 2, $z'=\varphi_{\theta(t_0+1+\tau_1)}(z)\neq\varphi_t(x_0)$, hence there is $n>0$ such that $\varphi_{-nT}(z')\in W^{ss}_{loc}(\varphi_{-nT}(\varphi_t(x_0))\setminus W^{ss}_{r}(\varphi_{-nT}(\varphi_t(x_0))$. This implies that $\varphi_{-nT}(z)$ is not in the $\varepsilon_0$ neighborhood of $\Gamma$. This contradicts the shadowing property and ends the proof of Proposition \ref{mainprop}.
\end{proof}

\paragraph{Proof of Theorem A}
Let $\Lambda$ be a singular hyperbolic chain transitive set of $X$ and $\sigma\in\Lambda$ be a singularity. Without loss of generality we assume that $\Lambda$ is positively singular hyperbolic with splitting $E^{ss}\oplus E^{cu}$. Assume $\Lambda$ has the intrinsic shadowing property. By Lemma \ref{lem6},  there is a point $x_0\in (W^s(\sigma)\cap W^u(\sigma)\cap\Lambda)\setminus\{\sigma\}$. It is easy to see that $\Gamma=Orb(x_0)\cup\{\sigma\}$ is singular hyperbolic and has the shadowing property, contradicting Proposition \ref{mainprop}. This ends the proof of Theorem A.

\section{The Proof of Theorem B}

In this section we prove that if a star flow $X$ has the shadowing property on its chain recurrent set, then $X$ satisfies Axiom A and no cycle condition.

Let us recall the definition of the linear Poincar\'e flow. For $x\in M\setminus Sing(X)$, set
$$N_x=N^X_x=\{v\in T_xM: v\perp X(x)\}.$$
Consider the bundle
$$N_{M\setminus Sing(X)}=N^X_{M\setminus Sing(X)}=\bigcup_{x\in M\setminus Sing(X)}N^X_x$$
over $M\setminus Sing(X)$. Then we define $$\psi_t=\psi^X_t: N_{M\setminus Sing(X)}\to N_{M\setminus Sing(X)}$$ by letting
$$\psi^X_t(v)=\Phi_t(v)-<v, X(\varphi_t(x))>\frac{X(\varphi_t(x))}{\|X(\varphi_t(x))\|}$$
for every $v\in N$. $\psi_t$ is called the {\it linear Poincar\'e flow } with respect to $X$. Let $\Lambda$ be a compact invariant set of $\varphi_t$. We say that $\Lambda\setminus Sing(X)$ admits a {\it dominated splitting} $N_{\Lambda\setminus Sing(X)}=\Delta^s\oplus\Delta^u$ of index $i$ with respect to $\psi_t$ if the following three conditions are satisfied:
\begin{enumerate}
\item  $\dim(\Delta^s(x))=i$ for all $x\in\Lambda\setminus Sing(X)$;
\item  $\Delta^s$ and $\Delta^u$ are $\psi_t$ invariant, that is $\psi_t(\Delta^s(x))=\Delta^s(\varphi_t(x))$ and $\psi_t(\Delta^u(x))=\Delta^u(\varphi_t(x))$ for all $x\in\Lambda\setminus Sing(X)$ and $t\in\mathbb{R}$;
\item there exist $C\geq 1$ and $\lambda>0$ such that $\|\psi_t|_{\Delta^s(x)}\|\cdot\|\psi_{-t}|_{\Delta^u(\varphi_t(x))}\|\leq Ce^{-\lambda t}$ for any $x\in\Lambda\setminus Sing(X)$ and $t\geq 0$.
\end{enumerate}
Since generally the normal spaces $N_x$ on regular points $x\in M\setminus Sing(X)$ can not be (in a unique way) extended to a singularity, the dominated splitting on $\Lambda\setminus Sing(X)$ can not be extended to the singularity, which causes a difficulty of non-compactness. The notion of the extended linear Poincar\'e flow introduced by Li-Gan-Wen\cite{LGW} then recovers the compactness by ``blowing up" the singularity as follows. Denote by
$$SM=\{e\in TM: \|e\|=1\}$$
the unit sphere bundle of $M$. Let $j:SM\to M$ be the natural projection defined by $j(e)=x$ for any $x\in M$ and $e\in T_xM\cap SM$. The tangent flow $\Phi_t$ induces a flow $\Phi_t^\#$ on $SM$ by $$\Phi_t^\#(e)=\Phi_t(e)/\|\Phi_t(e)\|.$$ At every $e\in SM$, define the normal space
$$N_e=\{v\in T_{j(x)}M: v\perp e\}.$$
Then we define the normal bundle $N_{SM}$ by attaching the linear space $N_e$ as the fiber of $e\in SM$. Note that for a given vector field $X$,  $N_{X(x)/\|X(x)\|}=N_x$. Then we extend the linear Poincar\'e flow $\psi_t$ of $X$ to a flow $$\tilde\psi_t=\tilde{\psi}^X_t:N_{SM}\to N_{SM}$$ by letting
$$\tilde\psi_t^X(v)=\Phi_t(v)-<v, \Phi_t^\#(e)>\Phi_t^\#(e), \text{ for every } v\in N_e.$$

Let $K\subset SM$ be a compact invariant set of $\Phi_t^\#$. We say that $K$ admits a {\it dominated splitting}  $N_{K}=\Delta^s\oplus\Delta^u$ of index $i$ with respect to $\tilde\psi_t$ if the following three conditions are satisfied:
\begin{enumerate}
\item  $\dim(\Delta^s(e))=i$ for all $e\in K$;
\item  $\Delta^s$ and $\Delta^u$ are $\tilde\psi_t$ invariant, that is, $\tilde\psi_t(\Delta^s(e))=\Delta^s(\Phi_t^\#(e))$ and $\tilde\psi_t(\Delta^u(e))=\Delta^u(\Phi^\#_t(e))$ for all $e\in K$ and $t\in\mathbb{R}$;
\item there exist $C\geq 1$ and $\lambda>0$ such that $\|\tilde\psi_t|_{\Delta^s(e)}\|\cdot\|\tilde\psi_{-t}|_{\Delta^u(\Phi^\#_t(e))}\|\leq Ce^{-\lambda t}$ for any $e\in K$ and $t\geq 0$.
\end{enumerate}

For a compact invariant set $\Lambda$ of $\varphi_t$, denote
$$\tilde\Lambda=\overline{\left\{\frac{X(x)}{\|X(x)\|}:x\in \Lambda\setminus Sing(X)\right\}},$$
where the closure is taken in the sphere bundle ${SM}$.
It is easy to see that if there is a dominated splitting $N_{\Lambda\setminus Sing(X)}=\Delta^s\oplus\Delta^u$ of index $i$ with respect to $\psi_t$, then there is a dominated splitting $N_{\tilde\Lambda}=\Delta^s\oplus\Delta^u$ of index $i$ with respect to $\tilde\psi_t$ on $\tilde\Lambda$ such that $\Delta^s(e)=\Delta^s(j(e))$ and $\Delta^u(e)=\Delta^u(j(e))$ for any $e\in j^{-1}(\Lambda\setminus Sing(X))\cap \tilde\Lambda$. Note that for simplicity we have used the same notation $\Delta^s\oplus\Delta^u$ for bundles of different bases.

Let $\sigma$ be a hyperbolic singularity of a vector field $X$. We can list all Lyapunov exponents of $\Phi_1|_{T_\sigma M}$ (counting the multiplicities) as
$$\lambda_1\leq\lambda_2\leq\cdots\lambda_s<0<\lambda_{s+1}\leq\cdots\leq\lambda_d.$$
The {\it saddle value} $sv(\sigma)$ of $\sigma$ with respect to $X$ is defined by
$$sv(\sigma)=\lambda_s+\lambda_{s+1}.$$
We say that $\sigma$ is {\it positively Lorenz-like} for $X$ if $\lambda_{s}$ has multiplicity $1$ (that is, $\lambda_{s-1}<\lambda_{s}$ )and $sv(\sigma)>0$.  We say that $\sigma$ is {\it Lorenz-like} for $X$ if $\sigma$ is either positively Lorenz-like for $X$ or positively Lorenz-like for $-X$.

Let $\sigma$ be a Lorenz-like singularity of $X$. If $sv(\sigma)>0$, we denote by $E^{c}(\sigma)$ the eigenspace of $\Phi_1|_{T_\sigma M}$ associated with the eigenvalue $e^{\lambda_{s}}$, and $E^{ss}(\sigma)$ the direct sum of all generalized eigenspaces of $\Phi_1|_{T_\sigma M}$ associated with the eigenvalues with norm less than $e^{\lambda_{s}}$. Denote $W^{ss}(\sigma)$ the strong stable manifold of $\sigma$ tangent to $E^{ss}(\sigma)$ at $\sigma$. Similarly, if $sv(\sigma)<0$, we denoted by $E^{c}(\sigma)$ the eigenspace of $\Phi_1|_{T_\sigma M}$ associated with the eigenvalue $e^{\lambda_{s+1}}$, and $E^{uu}(\sigma)$ the direct sum of all generalized eigenspaces of $\Phi_1|_{T_\sigma M}$ associated with the eigenvalues with norm bigger than $e^{\lambda_{s+1}}$. Also, denote $W^{uu}(\sigma)$ the strong unstable manifold of $\sigma$ tangent to $E^{uu}(\sigma)$ at $\sigma$.

Now we prove the following proposition.
\begin{Proposition}\label{singularhyperbolicity}
Let $X$ be a $C^1$ vector field and $\sigma$ be a positively Lorenz-like singularity. If there is a homoclinic point $x_0\in (W^s(\sigma)\cap W^u(\sigma))\setminus\{\sigma\}$ such that $x_0\notin W^{ss}(\sigma)$ and there is a dominated splitting $N_{Orb(x_0)}=\Delta^s\oplus\Delta^u$ of index $Ind(\sigma)-1$ with respect to $\psi_t$, then the compact invariant set $\Gamma=Orb(x_0)\cup\{\sigma\}$ is positively singular hyperbolic.
\end{Proposition}

\begin{proof}

By changing to an equivalent Riemannian metric, we can assume that the three bundles $E^{ss}(\sigma), E^c(\sigma), E^u(\sigma)$ are mutually orthogonal. Denote by $E^{cu}(\sigma)=E^c(\sigma)\oplus E^u(\sigma)$.

Let $\tilde\Gamma$ be the closure of the set
$$\left\{\frac{X(x)}{\|X(x)\|}:x\in Orb(x_0)\right\}$$
in $SM$. The dominated splitting $\Delta^s\oplus \Delta^u$ over $Orb(x_0)$ can be uniquely extended to a dominated splitting $\Delta^s\oplus \Delta ^u$ on $\tilde\Gamma$, that is, there is a $\tilde\psi_t$ invariant splitting $N_{\tilde\Gamma}=\Delta^s\oplus\Delta^u$ with constants $C\geq1, \lambda>0$ such that
$$\|\tilde\psi_t|_{\Delta^s(e)}\|\cdot\|\tilde\psi_{-t}|_{\Delta^u(\Phi_{-t}^\#(e))}\|<Ce^{-\lambda t}$$
for all $e\in \Gamma$ and $t\ge 0$. Moreover, we know that $\dim(\Delta^s(e))=Ind(\sigma)-1$ for every $e\in \tilde\Gamma$.

\begin{Claim}
There are $C\geq 1, \lambda>0$ such that
$$\frac{\|\tilde\psi_t|_{\Delta^s(e)}\|}{\|\Phi_t(e)\|}<Ce^{-\lambda t}, \ \ \ \  \|\tilde\psi_{-t}|_{\Delta^u(e)}\|\cdot\|\Phi_{-t}(e)\|<Ce^{-\lambda t},$$
for all $e\in\tilde\Gamma$ and $t>0$.
\end{Claim}

In fact, since $x_0\in W^s(\sigma)\setminus W^{ss}(\sigma)$, for any accumulation point $e$ of
$$\left\{\frac{X(\varphi_{t}(x_0))}{\|X(\varphi_{t}(x_0))\|}: t\geq 0\right\}$$
we have $e\in E^c(\sigma)$. By the invariance of the unstable manifold $W^u(\sigma)$, the accumulation points of
$$\left\{\frac{X(\varphi_{-t}(x_0))}{\|X(\varphi_{-t}(x_0))\|}: t\geq 0\right\}$$
are contained in $E^u(\sigma)$. Hence $\tilde\Gamma\cap T_\sigma M\subset E^c(\sigma)\oplus E^u(\sigma)$. By assumption $E^{ss}(\sigma), E^c(\sigma), E^u(\sigma) $ are mutually orthogonal, hence for any $e\in \tilde\Gamma\cap T_\sigma M$ one has $E^{ss}(\sigma)\subset N_e$. Thus there is a splitting $N_e=E^{ss}(\sigma)\oplus (N_e\cap (E^{cu}(\sigma)))$ for any $e\in\tilde\Gamma\cap T_\sigma M$. By the invariance of $E^{ss}(\sigma)$ and $E^{cu}(\sigma)$ under $\Phi_t$ and the invariance of $\tilde\Gamma\cap T_\sigma M$ under $\Phi_t^\#$, we know that
$\tilde\psi_t(E^{ss}(\sigma))=E^{ss}(\sigma)$ and $\tilde\psi_t(N_e\cap (E^{cu}(\sigma)))=N_{\Phi_t^\#(e)}\cap E^{cu}(\sigma)$. Since $E^{ss}(\sigma)\oplus E^{cu}(\sigma)$ is a dominated splitting w.r.t $\Phi_t$, for some big $T>0$ we will have
$$\|\Phi_t|_{E^{ss}(\sigma)}\|\cdot\|\Phi_{-t}|_{E^{cu}(\sigma)}\|<\frac{1}{2}$$
for all $t>T$. Since $\|\tilde\psi_t|_{E^{ss}(\sigma)\cap N_e}\|=\|\Phi_t|_{E^{ss}(\sigma)}\|$ and $\|\tilde\psi_t|_{E^{cu}(\sigma)\cap N_e}\|\leq\|\Phi_t|_{E^{cu}(\sigma)}\|$ for any $e\in \tilde\Gamma\cap T_\sigma M$ and any $t\in \mathbb R$,  it follows that
$$\|\tilde\psi_t|_{E^{ss}(\sigma)}\|\cdot\|\tilde\psi_{-t}|_{E^{cu}(\sigma)\cap N_{\Phi_t^\#(e)}}\|<\frac{1}{2}$$
for all $e\in \tilde\Gamma\cap T_\sigma M$ and $t>T$. By the fact that $\dim(E^{ss}(\sigma))=\dim(\Delta^s(e))=Ind(\sigma)-1$ and the uniqueness of dominated splitting we know that $E^{ss}(\sigma)=\Delta^s(e)$ and $E^{cu}(\sigma)\cap N_e=\Delta^u(e)$ for every $e\in \tilde\Gamma\cap T_\sigma M$. Since $\tilde\psi_t|_{\Delta^{s}(e)}=\tilde\psi_t|_{E^{ss}}=\Phi_t|_{E^{ss}(\sigma)}$ and since $E^{ss}(\sigma)\oplus E^{cu}(\sigma)$ is a dominated splitting w.r.t $\Phi_t$, one has
$$\frac{\|\psi_t|_{\Delta^s(e)}\|}{\|\Phi_t(e)\|}<\frac{1}{2}$$
for all $e\in T_\sigma M$ and $t>T$. By the continuity of $\Delta^s(e)$ we know that if $|s|$ is large enough then
$$\frac{\|\psi_t|_{\Delta^s(\varphi_s(x_0))}\|}{\|\Phi_t(\frac{X(\varphi_s(x_0))}{\|X(\varphi_s(x_0))})\|}<\frac{1}{2}$$
for any $t>T$. Hence for any $s\in\mathbb{R}$, we can find $T'$ sufficiently large such that the orbit segment $\varphi_{[0,T]}(\varphi_s(x_0))$ stays near $\sigma$ sufficiently long so that
$$\frac{\|\psi_{T'}|_{\Delta^s(\varphi_s(x_0))}\|}{\|\Phi_{T'}(\frac{X(\varphi_s(x_0))}{\|X(\varphi_s(x_0))})\|}<1.$$
This proves that for any $e\in\tilde\Gamma$, there is $T(e)$ such that
$$\frac{\|\psi_{T(e)}|_{\Delta^s(e)}\|}{\|\Phi_{T(e)}(e)\|}<1.$$
By the compactness of $\Gamma$ there are $C\geq 1, \lambda>0$ such that
$$\frac{\|\tilde\psi_t|_{\Delta^s(e)}\|}{\|\Phi_t(e)\|}<Ce^{-\lambda t}$$
for all $e\in\tilde\Gamma$ and $t>0$.

On the other hand, from the fact that $\Delta^u(e)=N_e\cap E^{cu}(\sigma)$ we know that
$$\|\tilde\psi_{-t}|_{\Delta^u(e)}\|\cdot\|\Phi_{-t}(e)\|\leq \|\wedge^2\Phi_{-t}|_{E^{cu}(\sigma)}\|$$
for every $e\in \Gamma\cap T_\sigma M$ and $t\in\mathbb{R}$. Since $sv(\sigma)>0$, every Lyapunov exponent of $\|\wedge^2\Phi_{-t}|_{E^{cu}(\sigma)}\|$ is strictly less than zero. Hence there is $T>0$ such that
$$\|\wedge^2\Phi_{-t}|_{E^{cu}(\sigma)}\|<\frac{1}{2}$$
for every $t>T$. Thus
$$\|\tilde\psi_{-t}|_{\Delta^u(e)}\|\cdot\|\Phi_{-t}(e)\|<\frac{1}{2}$$
for any $e\in\tilde\Gamma\cap T_\sigma M$ and $t>T$. Similarly by the continuity of $\Delta^u(e)$ we know that for
$|s|$ large one has
$$\|\psi_t|_{\Delta^s(\varphi_s(x_0))}\|\cdot\|\Phi_t(\frac{X(\varphi_s(x_0))}{\|X(\varphi_s(x_0))})\|<\frac{1}{2}$$
for any $t>T$. Hence for any $s\in\mathbb{R}$, we can find $t>0$ large enough such that $\varphi_{[0,t]}(\varphi_s(x_0))$ stay near $\sigma$ long enough so that
$$\|\psi_t|_{\Delta^s(\varphi_s(x_0))}\|\cdot\|\Phi_t(\frac{X(\varphi_s(x_0))}{\|X(\varphi_s(x_0))})\|<1.$$
By the compactness of $\Gamma$ there are $C\geq 1, \lambda>0$ such that
$$\|\tilde\psi_t|_{\Delta^s(e)}\|\cdot\|\Phi_t(e)\|<Ce^{-\lambda t}$$
for all $e\in\tilde\Gamma$ and $t>0$. This ends the proof of the Claim.
\vskip 0.2cm

The inequality $$\frac{\|\tilde\psi_t|_{\Delta^s(e)}\|}{\|\Phi_t(e)\|}<Ce^{-\lambda t}$$
tells that $\Delta^s$ dominates the flow direction. It is a ``mixed'' domination as $\Delta^s$ is invariant under $\psi_t$ while the flow direction is invariant under $\Phi_t$. Now we can finish the proof of the lemma in the same way as Proposition 4.5 of \cite{WWY} does. We just give a sketch.

Let $E^{cu}(\varphi_t(x_0))=<X(\varphi_t(x_0))>\oplus \Delta^u(\varphi_t(x_0)$. This gives a  subbundle $E^{cu}$ over $\Gamma$. Since
$$\|\tilde\psi_{-t}|_{\Delta^u(e)}\|\cdot\|\Phi_{-t}(e)\|<Ce^{-\lambda t}$$
for all $e\in\tilde\Gamma$ and all $t\in\mathbb{R}$, one sees that
$$\|\psi_{-t}|_{\Delta^u(\varphi_t(x_0))}\|\cdot\|\Phi_{-t}|_{<X(\varphi_t(x_0))>}\|<Ce^{-\lambda t}$$
for all $t\in\mathbb{R}$.  Denote by $\eta_1$ and $\eta_2$ the largest two Lyapunov exponents of $\Phi_{-1}|_{E^{cu}}$, then  $\eta_1$ has multiplicity 1 and $\eta_1+\eta_2<0$, and then $E^{cu}$ is area expanding by Proposition 3.4 of \cite{WWY}.

On the other hand, since
$$\frac{\|\tilde\psi_t|_{\Delta^s(e)}\|}{\|\Phi_t(e)\|}<Ce^{-\lambda t}$$
for all $e\in\tilde\Gamma$, by applying Lemma 5.6 of \cite{LGW} we can find a $\Phi_t$ invariant bundle $E^{ss}$ that matches $E^{ss}(\sigma)$ at the singularity $\sigma$. Since $\Delta^s\oplus<X>$ and $\Delta^s\oplus\Delta^u$ are dominated splittings one can check that $E^{ss}\oplus E^{cu}$ is a dominated splitting and  $E^{ss}$ is contracting with respect to $\Phi_t$. This proves that $\Gamma$ is positively singular hyperbolic, and completes the proof of Proposition \ref{singularhyperbolicity}.
\end{proof}

The following property of star flows was proven in \cite{SGW}.
\begin{Proposition}[\cite{SGW} Lemma 4.2, Corollary 4.3 ]\label{saddlevalue}
Let $X$ be a star vector field and $\sigma$ be a singularity of $X$. If $C(\sigma)$ is nontrivial, then $\sigma$ is Lorenz-like for $X$.
\end{Proposition}

Let $\sigma$ be a Lorenz-like singularity. If $sv(\sigma)>0$, we set $Ind_p(\sigma)=Ind(\sigma)-1$, and if $sv(\sigma)<0$, we set $Ind_p(\sigma)=Ind(\sigma)$. This notion of $Ind_p(\sigma)$ is referred to as the {\it periodic index} of $\sigma$ in \cite{SGW}, which tells the possible index of periodic orbits arising from perturbations of a homoclinic loop of $\sigma$, as stated in the next Proposition which is  Lemma 4.4 of \cite{SGW}.

\begin{Proposition}\label{Periodindex}
Let $X$ be a star vector field and $\sigma$ be a singularity. If there is a homoclinic point $x_0\in (W^s(\sigma)\cap W^u(\sigma))\setminus\{\sigma\}$, together with a sequence $X_n$ and a sequence of periodic orbits $\gamma_n$ of $X_n$ such that $X_n\to X$ and $\gamma_n\to (Orb(x_0)\cup\{\sigma\})$ as $n\to\infty$, then  $Ind(\gamma_n)=Ind_p(\sigma)$ for large $n$.
\end{Proposition}

The following proposition is Theorem 3.6 of \cite{SGW}.

\begin{Proposition}\label{Lorenz-like}
Let $X$ be a star vector field and $\sigma$ be a singularity with $C(\sigma)$ nontrivial. If $sv(\sigma)>0$, then $(W^{ss}(\sigma)\setminus\{\sigma\})\cap C(\sigma)=\emptyset$. If $sv(\sigma)<0$, then $(W^{uu}(\sigma)\setminus\{\sigma\})\cap C(\sigma)=\emptyset$.
\end{Proposition}

\begin{Lemma}\label{lem35}
Assume that $X$ is a star vector field that has the shadowing property on $CR(X)$. If there is a singularity $\sigma$ with $C(\sigma)$ nontrivial, then there is a point $x_0\in (W^s(\sigma)\cap W^u(\sigma))\setminus\{\sigma\}$ such that the compact invariant set $\Gamma=Orb(x_0)\cup\{\sigma\}$ is singular hyperbolic.
\end{Lemma}

\begin{proof}
By Proposition \ref{saddlevalue} we know that $\sigma$ is Lorenz-like. With out loss of generality we assume that $\sigma$ is positively Lorenz-like. By Lemma \ref{lem6}, there is $x_0\in (W^s(\sigma)\cap W^u(\sigma))\setminus \{\sigma\} $ as $C(\sigma)$ is chain transitive. We show that $\Gamma=Orb(x_0)\cup\{\sigma\}$ is positively singular hyperbolic.

It is easy to see that $\Gamma\subset C(\sigma)$. By Proposition \ref{Lorenz-like} we know that $x_0\notin W^{ss}(\sigma)$. By a perturbation near the singularity, we can get a sequence $X_n\to X$ with a periodic orbit $\gamma_n$ of $X_n$ such that $\gamma_n \to\Gamma $ in the Hausdorff metric as $n\to\infty$. By Proposition \ref{Periodindex} we know that for $n$ large, the index of $\gamma_n$ is $Ind(\sigma)-1$. For any $t\in\mathbb{R}$, we can find a sequence of $p_n\in\gamma_n$ with $p_n\to \varphi_t(x_0)$ such that $N^s_{p_n}$ accumulate on $\Delta^s_{\varphi_t(x_)}$ and $N^u_{p_n}$ accumulate on $\Delta^u_{\varphi_t(x_)}$ as $n\to\infty$. Since $N^s_{p_n}\oplus N^u_{p_n}$ are uniformly dominated \cite{Liao}, the splitting $N_{Orb(x_0)}=\Delta^s\oplus\Delta^u$ is a dominated splitting of index $Ind(\sigma)-1$ with respect to $\psi_t$. By Proposition \ref{singularhyperbolicity}, $\Gamma$ is a positively singular hyperbolic set.
\end{proof}

A point $x$ is called a {\it preperiodic point} of $X$ if there is a sequence $X_n\to X$ together with periodic points $p_n$ of $X_n$ such that $p_n\to x$ as $n\to\infty$. Denote by $P_*(X)$ the set of all preperiodic points of $X$. The following proposition is Theorem A' of \cite{GW}.
\begin{Proposition}
If $X$ is a star vector field and $P_*(X)$ contains no singularity, then $X$ satifies Axiom A and the no-cycle condition.
\end{Proposition}

Now we prove the main theorem of this section.

\vskip 0.2cm
{\bf Theorem B.}
If $X$ is a star vector field that has the shadowing property on $CR(X)$, then $X$ satisfies Axiom A and the no-cycle condition.
\vskip 0.2cm
\begin{proof}
Assume $X$ is a star vector field that has the shadowing property on $CR(X)$. It suffices to prove that $P_*(X)$ contains no singularity. Assume the contrary that there is a singularity $\sigma\in P_*(X)$. There is a sequence $X_n$ together with  a periodic orbit $\gamma_n$ of $X_n$ and $p_n\in\gamma_n$ with $\lim_{n\to\infty }p_n=\sigma$. Without loss of generality we assume that $\gamma_n$ has a limit in the Hausdorff metric. Since $\sigma$ is hyperbolic, the Hausdorff limit of $\gamma_n$ can not be $\{\sigma\}$. Hence there are points $q_n\in\gamma_n$ such that $q_n\to y\neq\sigma$ as $n\to\infty$. For any $\delta>0$ we can find $N$ such that $\gamma_n$ is a $\delta$-pseudo orbit of $X$ for $n\geq N$. Hence $y$ is contained in the chain class $C(\sigma)$. This proves that $C(\sigma)$ is nontrivial. By Proposition \ref{lem35} there is a homoclinic point $x_0\in (W^s(\sigma)\cap W^u(\sigma))\setminus\{\sigma\}$ and $\Gamma=Orb(x_0)\cup\{\sigma\}$ is a singular hyperbolic set. Since $\Gamma\subset CR(X)$, $\Gamma$ has the shadowing property. This contradicts Proposition \ref{mainprop}, completing the proof of Theorem B.
\end{proof}

\section{The Proof of Theorem C}
In this section we prove that if a multisingular hyperbolic chain transitive set has the intrinsic shadowing property then it is hyperbolic. First we recall the definition of multisingular hyperbolicity proposed by Bonatti-da Luz \cite{Bl}.

Let $\Lambda$ be a compact invariant set of a $C^1$ vector field $X$ containing no nonhyperbolic singularity. Let $\sigma\in\Lambda$ be a hyperbolic singularity of $X$. As usual, we list all Lyapunov exponents of $\Phi_1|_{T_\sigma M}$ (counting the multiplicities) as
$$\lambda_1\leq\lambda_2\leq\cdots\lambda_s<0<\lambda_{s+1}\leq\cdots\leq\lambda_d.$$
Let $j\leq s$ be the largest natural number with the following two properties:
\begin{itemize}
\item[(i)] $\lambda_j<\lambda_{j+1}$;
\item[(ii)] $W^{es}(\sigma)\cap\Lambda=\{\sigma\}$, where $W^{es}(\sigma)$ is the strong stable mainifold tangent to the direct sum of the generalized eigenspaces of eigenvalues $\eta$ with $|\eta|<e^{\lambda_{j+1}}$.
\end{itemize}
We call the direct sum of the generalized eigenspaces of eigenvalues $\eta$ with $|\eta|<e^{\lambda_{j+1}}$ the {\it escaping strong stable spaces } of $\sigma$ with respect to $\Lambda$ and denote it by $E^{es}(\sigma)$. Similarly we define the {\it escaping strong unstable spaces} $E^{eu}(\sigma)$ of $\sigma$ with respect to $\Lambda$. We also denote by $E^{ec}(\sigma)$ the direct sum of the other generalized eigenspaces of $\Phi_1|_{T_\sigma M}$. It is easy to see that the set
$$B(\Lambda)=\{\frac{X(x)}{\|X(x)\|}: x\in\Lambda\setminus Sing(X)\}\cup(\bigcup_{\sigma\in\Lambda\cap Sing(X)}\{v\in E^{ec}(\sigma):\|v\|=1\})$$
is a compact invariant set of $\Phi_t^\#$.

Since $\Lambda$ contains only finitely many hyperbolic singularities $\{\sigma_1, \sigma_2, \cdots, \sigma_n\}$, there are mutually disjoint open sets $U_1, U_2, \cdots, U_n$ in $SM$ such that $B(\Lambda)\cap T_{\sigma_i} M\subset U_i$ for every $i=1,2,\cdots, n$. Given any $i=1,2,\cdots,n$, Lemma 41 of \cite{Bl} says that there exists (an unique) $h^{\sigma_i}: B(\Lambda)\times\mathbb{R}\to (0,+\infty)$ with the following properties:
\begin{itemize}
\item[(a)] $h^{\sigma_i}(e, t+s)=h^{\sigma_i}(e,t)\cdot h^{\sigma_i}(\Phi_t^\#(e),s)$ for any $e\in B(\Lambda)$ and $t,s\in\mathbb{R}$;
\item[(b)] if $e\in U_i$ and $\Phi_t^\#(e)\in U_i$ then $h^{\sigma_i}(e, t)=\|\Phi_t(e)\|$;
\item[(c)] if $e\notin U_i$ and $\Phi_t^\#(e)\notin U_i$, then $h^{\sigma_i}(e, t)=1$.
\end{itemize}
As usual, we frequently write $h^{\sigma_i}(e,t)$ as $h^{\sigma_i}_t(e)$.

\begin{Definition}
Let $\Lambda$ be a compact invariant set of a $C^1$ vector field. We say that $\Lambda$ is multisingular hyperbolic if
\begin{enumerate}
\item Every singularity in $\Lambda$ is hyperbolic.
\item There is a continuous dominated splitting $N_{B(\Lambda)}=\Delta^s\oplus\Delta^u$ with respect to $\tilde\psi_t$.
\item There exist two sets $S_+, S_{-}\subset \Lambda\cap Sing(X)$ and constants $C\geq 1, \lambda>0$, such that $(\displaystyle\prod_{\sigma\in S_+}h^\sigma_t)\cdot\tilde\psi_t|_{\Delta^s}$  and $(\displaystyle\prod_{\sigma\in S_-}h^\sigma_{-t})\cdot\tilde\psi_{-t}|_{\Delta^u}$ are $(C,\lambda)$-contracting, that is,
$$\|(\displaystyle\prod_{\sigma\in S_+}h^\sigma_t)\cdot\tilde\psi_t|_{\Delta^s(e)}\|\leq Ce^{-\lambda t},$$
$$\|(\displaystyle\prod_{\sigma\in S_-}h^\sigma_{-t})\cdot\tilde\psi_{-t}|_{\Delta^u(e)}\|\leq Ce^{-\lambda t}$$
for all $e\in B(\Lambda)$ and $t\geq 0$. Here $S_+$ (or $S_-$) could be empty, and in that case we will set $\displaystyle\prod_{\sigma\in S_+}h^\sigma_t\equiv1$ (or $\displaystyle\prod_{\sigma\in S_-}h^\sigma_{-t}\equiv 1$ respectively).
\end{enumerate}
\end{Definition}

If $\dim(\Delta^s(e))=i$ for all $e\in B(\Lambda)$, we will call $\Lambda$  a multisingular hyperbolic set {\it of index} $i$.
The following lemma is Proposition 56 of \cite{Bl}.
\begin{Lemma}\label{lem42}
Let $\Lambda$ be a multisingular hyperbolic set of $X$ of index $i$. If a singularity $\sigma\in\Lambda$ satisfies $(W^s(\sigma)\setminus\{\sigma\})\cap\Lambda\neq\emptyset$ and $(W^u(\sigma)\setminus\{\sigma\})\cap\Lambda\neq\emptyset$, then $\sigma$ is Lorenz-like and $Ind_p(\sigma)=i$.
\end{Lemma}

The following lemma is taken from  Crovisier-da Luz-Yang-Zhang \cite{ClYZ}. For completeness we give a sketch of the  proof here.
\begin{Lemma}\label{lem43}
Let $\Lambda$ be a multisingular hyperbolic set and $\sigma\in\Lambda$ be a hyperbolic singularity with $(W^{s}(\sigma)\setminus\{\sigma\})\cap\Lambda\neq\emptyset$ and $(W^{u}(\sigma)\setminus\{\sigma\})\cap\Lambda\neq\emptyset$. If $\sigma$ is positively Lorenz-like then $W^{ss}(\sigma)\cap\Lambda=\{\sigma\}$.
\end{Lemma}

\begin{proof}
Assume the contrary that $(W^{ss}(\sigma)\setminus\{\sigma\})\cap\Lambda\neq\emptyset$. Then by the definition of escaping strong stable manifold we know that $E^{ec}(\sigma)\cap E^{ss}(\sigma)\neq\emptyset$. Likewise, from $(W^{u}(\sigma)\setminus\{\sigma\})\cap\Lambda\neq\emptyset$ we obtain $E^{ec}(\sigma)\cap E^{u}(\sigma)\neq\emptyset$. Take a unit vector  $e^{ss}\in E^{ec}(\sigma)\cap E^{ss}(\sigma)$ and a unit vector $e^{u}\in E^{ec}(\sigma)\cap E^{u}(\sigma)$, and let   $$e=\frac{e^{ss}+e^{u}}{\|e^{ss}+e^{u}\|}\in E^{ec}.$$   It is easy to see that $\Phi_t^\#(e)$ accumulate on $E^{u}(\sigma)$ as $t\to+\infty$ and on $E^{ss}(\sigma)$ as $t\to-\infty$. Without loss of generality we assume that $E^{ss}, E^{c}, E^{u}$ are mutually orthogonal. Let $e'$ be an accumulation point of $\{\Phi_t^\#(e):t\geq0\}$. By the uniqueness of dominated splitting, $\Delta^s(e')=E^{ss}(\sigma)$. Similarly, if $e''$ is an accumulation point of $\{\Phi_t^\#(e):t\leq 0\}$ then $\Delta^u(e'')=E^{u}(\sigma)$.   Let $N_e=\Delta^s(e)\oplus\Delta^u(e)$ be the dominated splitting at $e$.  Take the vector $v=e^{ss}+e^{u}\in N_e$. An easy calculation shows that
$$\tilde\psi_t(v)=2\frac{\|\Phi_t(e^u)\|^2 \Phi_t(e^{ss})-\|\Phi_t(e^{ss})\|^2\Phi_t(e^u)}{\|\Phi_t(e^{ss})\|^2+\|\Phi_t(e^u)\|^2}.$$
We can see that the direction of $\tilde\psi_t(v)$ tends to $E^{ss}(\sigma)$ as $t\to+\infty$ and tends to $E^{u}(\sigma)$ as $t\to-\infty$. If $v\in\Delta^s(e)$, then $\tilde\psi_t(v)$ will tend to the direction of $E^{ss}(\sigma)$ as $t\to-\infty$ by the invariance of $\Delta^s$. If $v\notin\Delta^s(e)$, then $\tilde\psi_t(v)$ will tend to the direction of $\Delta^u(\Phi_t^\#(e))$, hence to $E^{u}(\sigma)$ as $t\to+\infty$. In both cases we get a contradiction. Hence $W^{ss}(\sigma)\cap\Lambda=\{\sigma\}$.
\end{proof}

Now we prove the main theorem of this section.
\vskip 0.2cm
{\bf Theorem C.}
Let $\Lambda$ be a nontrivial multisingular hyperbolic chain transitive set. If $\Lambda$ has the intrinsic shadowing property, then $\Lambda$ contains no singularity.
\vskip 0.2cm
\begin{proof}
Let $N_{B(\Lambda)}=\Delta^s\oplus\Delta^u$ be the dominated splitting in the definition of multisingular hyperbolicity. Since $\Lambda$ is chain transitive we know that there is $i$ such that $\Delta^s(e)=i$ for all $e\in B(\Lambda)$.

Suppose there is a singularity $\sigma$ in $\Lambda$. Since $\Lambda$ is chain transitive, by Lemma \ref{lem1} we obtain $W^s(\sigma)\setminus\{\sigma\}\neq\emptyset$ and $W^u(\sigma)\setminus\{\sigma\}\neq\emptyset$. By Proposition \ref{lem42} we know that $\sigma$ is Lorenz-like. Without loss of generality we assume $\sigma$ is positively Lorenz-like. Since $\Lambda$  has the intrinsic shadowing property, by applying Lemma \ref{lem6}, there is a point $x_0\in (W^s(\sigma)\cap W^u(\sigma)\cap\Lambda)\setminus\{\sigma\}$. By the choice of $U$, there is a dominated splitting $N_{Orb(x_0)}=\Delta^s\oplus\Delta^u$ of index $i$ with respect to $\psi_t$.  Note that $Ind(\sigma)-1=i$ by Proposition \ref{lem42}. By Lemma \ref{lem43},  $W^{ss}(\sigma)\cap\Lambda=\{\sigma\}$. Note that there is a dominated splitting $N_{Orb(x_0)}=\Delta^s\oplus\Delta^u$ from the dominated splitting on $B(\Lambda)$. By Proposition \ref{singularhyperbolicity}, $\Gamma=Orb(x_0)\cup\{\sigma\}$ is positively singular hyperbolic. Since $\Lambda$ has the intrinsic shadowing property, it follows that $\Gamma$ has the shadowing property. This contradicts Proposition \ref{mainprop}, proving Theorem C.
\end{proof}

\end{document}